\documentclass[12pt]{amsart}
\usepackage{amsthm,stmaryrd}
\usepackage{bbm}
\usepackage{amssymb,enumerate}
\usepackage{graphicx}
\usepackage{epsfig}
\usepackage[usenames,dvipsnames]{color}
\usepackage{verbatim}
\usepackage{hyperref}
\usepackage{mathrsfs}
\usepackage{dsfont}
\usepackage{dsfont}
\usepackage[all]{xy}
\usepackage{enumitem}
\newcommand{\Png}[2]{
{\includegraphics[height=#1]{#2}}
}
\newcommand{\coev}{\operatorname{coev}}
\newcommand{\ev}{\operatorname{ev}}
\newcommand{\tev}{\widetilde{\operatorname{ev}}}
\newcommand{\tcoev}{\widetilde{\operatorname{coev}}}

\newcommand{\rt}{\mathsf{p}}
\newcommand{\slt}{{\mathfrak{sl}(2)}}
\newcommand{\Urest}{{{U}}}
\newcommand{\repX}{\mathscr{X}}
\newcommand{\Hlog}{\mathrm{H}^{\mathrm{log}}}
\newcommand{\Jun}{\mathrm{J}^{\mathrm{log}}}

\newcommand{\U}{{U}}

\newcommand{\UD}{{D}}

\newcommand{\mTr}{\mathrm{Tr}'}
\renewcommand{\t}{{\mathsf{t}}}
\newcommand{\rTr}{{\mathsf{t}}^R}
\newcommand{\lTr}{{\mathsf{t}}^L}

\newcommand{\tr}{\operatorname{tr}}
\newcommand{\cat}{\mathscr{C}}

\newcommand{\Id}{\operatorname{id}}
\newcommand{\Proj}{\mathcal{P}}

\DeclareMathOperator{\Upmod}{{\it U}-pmod}
\DeclareMathOperator{\Umod}{{\it U}-mod}

\DeclareMathOperator{\Dmod}{{\it D}-mod}

\newcommand{\cAlg}{\mathfrak{C}}

\newcommand{\SL}{\ensuremath{{\mathfrak{sl}}(2)}}
\newcommand{\ideme}{{\textbf{e}}}
\newcommand{\idemw}{{\textbf{w}}}
\newcommand{\one}{\mathbf{1}}
\newcommand{\tclass}{\textbf{h}}

\newcommand{\junyp}{\textbf{y}^+}
\newcommand{\junxp}{\textbf{x}^+}
\newcommand{\junbp}{\textbf{b}^+}
\newcommand{\junap}{\textbf{a}^+}

\newcommand{\junym}{\textbf{y}^-}
\newcommand{\junxm}{\textbf{x}^-}
\newcommand{\junbm}{\textbf{b}^-}
\newcommand{\junam}{\textbf{a}^-}

\newcommand{\junypm}{\textbf{y}^\pm}
\newcommand{\junxpm}{\textbf{x}^\pm}
\newcommand{\junbpm}{\textbf{b}^\pm}
\newcommand{\junapm}{\textbf{a}^\pm}

\newcommand{\qv}{{ r}}

\newcommand{\End}{\operatorname{End}}

\newcommand{\Hom}{\operatorname{Hom}}

\newcommand{\HH}{\operatorname{HH}_0}
\newcommand{\Ob}{\operatorname{\Ob}}

\newcommand{\C}{\ensuremath{\mathbb{C}} }
\newcommand{\Z}{\ensuremath{\mathbb{Z}} }

\newcommand{\Rib}{\ensuremath{\mathrm{Rib}} }
\newcommand{\qCh}{\ensuremath{\mathrm{qChar}} }
\newcommand{\Ch}{\ensuremath{\mathrm{Char}} }

\newcommand{\PP}{\ensuremath{\mathcal{P}} }

\newtheorem{definition}{Definition}
\newtheorem{theorem}[definition]{Theorem}

\newtheorem{proposition}[definition]{Proposition}
\newtheorem{lemma}[definition]{Lemma}

\newcounter{IntroCounter}
\stepcounter{IntroCounter}

\theoremstyle{remark}
\newtheorem*{remark}
{Remark}
\newcounter{exo} \newcounter{numexercice}
\renewcommand{\theexo}{\arabic{exo}}

\newcommand{\epsh}[2]
         {\begin{array}{c} \hspace{-1.3mm}
        \raisebox{-4pt}{\epsfig{figure=#1.pdf,height=#2}}
        \hspace{-1.9mm}\end{array}}

\begin{document}
\title[Logarithmic Hennings invariant ]{Logarithmic Hennings invariants\\[.2cm] for restricted quantum
$\SL$}

\author[A. Beliakova]{Anna Beliakova}
\address[anna@math.uzh.ch]{Anna Beliakova, Institut f\"ur Mathematik,
Universit\"at Z\"urich,
Winterthurerstrasse 190,
CH-8057 Z\"urich.}

\author[C. Blanchet]{Christian Blanchet}
\address[christian.blanchet@imj-prg.fr]{Christian Blanchet, Univ Paris Diderot, Sorbonne Paris Cit\'e, IMJ-PRG, UMR 7586 CNRS,  F-75013, Paris.}

\author[N. Geer]{Nathan Geer}
\address[nathan.geer@gmail.com]{Nathan Geer, Mathematics \& Statistics\\ Utah State University \\ Logan, Utah 84322, USA}

\begin{abstract}
We construct a Hennings type logarithmic invariant
for restricted quantum $\SL$ at a $2\rt$-th root of unity.
This quantum group $\Urest$ is not braided, but factorizable.
The invariant is defined for a pair: a 3-manifold $M$ and
a colored link $L$ inside $M$.  
The link $L$ is split into two parts colored 
 by central elements and  by trace classes, or elements in the $0^{\text{th}}$ Hochschild homology
of $\Urest$, respectively.
The two main ingredients of our construction are
 the universal
invariant of a string link with values in tensor powers of $\Urest$, and
the  {\it modified trace} introduced by the third author with
his collaborators and
computed on tensor powers of  the regular representation.  
 Our invariant is a colored extension  of the  logarithmic invariant constructed by Jun Murakami.

\end{abstract}

\maketitle
\setcounter{tocdepth}{3}
\section{Introduction}
In the 90th M. Hennings came up with a  construction of 
3-manifold invariants out of   a factorizable ribbon 
Hopf algebra $H$ \cite{Hennings}. In his construction 
the right integral $\mu\in H^*$ satisfying
$$ (\mu \otimes {\rm{id}}) \Delta(x)=\mu(x) \one \quad \text{for all}\quad x\in H$$
plays the role of the Kirby color.
If the category of $H$-modules is semi-simple, Hennings recovers
 the Reshetikhin--Turaev invariant.
However in the non semi-simple case,  his invariant vanishes
for manifolds with positive first Betti 
number (see  \cite{Kerler_genealogy}). 
%
 A TQFT based on the Hennings invariant
was constructed by 
Lyubashenko and Kerler \cite{Lyu, Kerler_MCG}. 
It satisfies the full TQFT axioms for 
lagrangian cobordisms between connected surfaces
with one boundary component. In the general case, it has 
 weak functoriality and monoidality properties.
%
%

More recently,  a completely different non-semisimple TQFT
 based on the unrolled quantum $\SL$ was defined by
 Blanchet, Costantino, Geer and Patureau \cite{BCGP}. This construction
 uses  the logarithmic  3-manifold invariant  constructed previously by Costantino, Geer and Patureau (CGP) 
 in \cite{CGP1}
 by generalizing
  the Kashaev invariant. More precisely,
 the CGP invariant is defined for an admissible pair: a 3-manifold 
 together with a $\C/2\Z$ valued cohomology class.
 One of the most innovative ingredients of  the CGP construction is 
 the so-called {\it modified trace} which contrary to  the 
 usual quantum trace does not vanish on the projective modules. 
 
 In this paper we construct a family of invariants of $3$-manifolds
 with colored links inside by combining Hennings 
 approach 
 with the modified trace methods  of \cite{GKP2,GPV}. 
 We develop structural properties of
  restricted quantum $\SL$,
 by working in the Hopf algebra itself rather than in its category
  of modules.  Our invariants generalize the logarithmic invariants of knots in 3-manifolds recently
  introduced by J. Murakami \cite{jM2013}. 
\subsection*{Main results}  
 Let us denote by $\Urest$ the restricted quantum $\SL$ at
$2\rt^{\text{th}}$ root of unity $q=e^{\frac{i\pi}{\rt}}$,
explicitely defined in the next section.
The Hopf algebra $\Urest$ does not contain
 an $R$-matrix, but only monodromy or double braiding.
However, $\Urest$ 
  is a subalgebra of a ribbon
  Hopf algebra $\UD$ obtained by adjoining
 a square root of the generator $K$.
 Since $\UD$ is not factorizable and $\Urest$ is not braided, 
 neither $\Urest$, nor $\UD$ supports the Hennings--Kerler--Lyubashenko
 TQFT construction. 

Let $\Umod$ be the category of finite dimensional $\Urest$-modules.
This is a finite pivotal tensor category. Hence, for any morphism
$f$ in $\Umod$ there is a notion of a categorical (or quantum) left  and right traces, denoted by $\tr_l(f)$ and $\tr_r(f)$, respectively.

 Let $\Upmod$ be its full subcategory of projective $\Urest$-modules. An explicit description of $\Upmod$ is given
  in \cite{FGST2}. 
  Let us denote by $\PP^\pm_j$ with $j=1,...,\rt$ the indecomposable projective modules.  Here $\PP^\pm_\rt$ is a simple module with highest weight $\pm q^{\rt-1}$. The module $\PP^+_1$ is the projective cover of the trivial module.  The space
  of endomorphisms $\End_\Urest(\PP^\pm_j)$ ($1\leq j\leq \rt-1$) is two dimensional
  with basis given by the identity $\Id_{\PP^\pm_j}$ and a 
  nilpotent endomorphism $x^\pm_j$, defined in Section 2.

The subcategory $\Upmod$ is an ideal of $\Umod$ in the sense of \cite{GKP2, GPV}.
A {\em modified trace} on $\Upmod$, is  a family of linear functions
$$\{\t_V:\End_\Urest V\rightarrow \C \}_{V\in \Upmod}$$
such that the following two conditions hold:
\begin{enumerate}
\item {\bf Cyclicity.}  If $X,V\in \Upmod$, then for any morphisms \mbox{$f:V\rightarrow X $} and $g:X\rightarrow V$  in $\Urest$-mod we have 
\begin{equation*}\label{E:Tracefggf}
\t_V(g f)=\t_X(f  g).
\end{equation*} 
\item {\bf Partial trace properties.} If $X\in \Upmod$ and $W\in \Umod$ then for any $f\in \End_\Urest(X\otimes W)$ and $g\in \End_\Urest(W\otimes X)$ we have
\begin{align*}%
\t_{X\otimes W}\left(f \right)
&=\t_X \left( \tr_r^W(f)\right),  \\
 \t_{W\otimes X}\left(g \right)
&=\t_X \left( \tr_l^W(g)\right) , 
\end{align*}
where $\tr_r^W$ and $\tr_l^W$ are the right and left partial categorical traces along $W$ defined using the pivotal structure in Eqs. \eqref{E:PartialLRtrace}.  
\end{enumerate}

Our first main result is the following.

\begin{theorem}\label{mTrace}
There exists a unique  family of linear functions
$$\{\t_V:\End_\Urest(V)\rightarrow \C \}_{V\in \Upmod}\ ,$$
satisfying cyclicity and the partial trace properties,
 normalized by 
$$\t_{\PP^+_\rt}(\Id_{\PP^+_\rt})=(-1)^{\rt-1}\ .$$
Moreover, $\t_{\PP^-_\rt}(\Id_{\PP^-_\rt})=1$ and for $1\leq j\leq \rt-1$ we have
$$\t_{\PP^\pm_j}(\Id_{\PP^\pm_j})=
(\pm 1)^{\rt-1} (-1)^j(q^j +q^{-j})\quad\text{and}\quad
\t_{\PP^\pm_j}(x^\pm_j)=(\pm 1)^{\rt}(-1)^j[j]^2 \, .
$$

\end{theorem}
This family is called the modified trace on $\Upmod$.
The proof  uses the fact that $\Umod$ is unimodular (i.e. projective cover of the trivial module is self-dual)
and it has a simple projective object. In this case, 
there exist unique  left and right modified traces on $\Upmod$
by \cite[Cor. 3.2.1]{GKP2}.
We actually compute these traces on the algebra of endomorphisms
of indecomposable projectives and show that they are equal.


Observe that Theorem \ref{mTrace}  applies to
  $\Urest$, considered as a free left module over itself,
  called  the regular representation,
   and its tensor powers. 
 Recall from \cite{FGST2} 
that the regular representation decomposes as
 $$\Urest \cong \bigoplus_{j=1}^{\rt} j \PP_j^+ \oplus 
\bigoplus_{j=1}^{\rt} j \PP_j^- .
$$
The algebra $\End_\Urest(\Urest)$ of the $\Urest$-endomorphisms of
$\Urest$ can be identified with $\Urest^{op}$ (i.e. $\Urest$ with the opposite multiplication). The isomorphism is given by sending an element $x$ of $\Urest^{op}$ to the operator $r_x$ of the right multiplication  by $x$ 
on the regular representation. By definition $r_x$ commutes with the left action.
 More generally, 
  $\cAlg_m=\End_\Urest(\Urest^{\otimes m})$, $m\geq 1$, are 
  known as  centralizer algebras. 

The space of characters (or symmetric functions, see \cite{Arike}) on $\Urest$ is defined as
$$\Ch(\Urest):=\{\phi \in \Urest^* \;|\; \phi ( xy)=\phi(yx) \quad\text{for any}\quad x,y\in \Urest\} .$$ 
This space is dual to 
the  
 $0^{\text{th}}$-Hochschild homology $\HH(\Urest)$, which
 is 
$$ \HH(\Urest):=\frac{U}{[U,U]} \quad\text{where}\quad
[\Urest,\Urest]={\rm{Span}}\{xy-yx\;|\;  x,y\in \Urest\} .$$
 There is an obvious action of the center $Z(\Urest)$ on
 $\Ch(\Urest)$ by setting $z\phi(x):=\phi(zx)$ for any $z\in Z(U)$ and $x\in U$.

 Let us define the linear map
 $$\mTr_m :\cAlg_m\rightarrow \C\quad {\text{by}}\quad
 x\mapsto t_{\Urest^{\otimes m}}(x)$$
 and in particular,  $\mTr: \Urest^{op}\simeq \cAlg_1\rightarrow \C$
is the linear map $\mTr_1=t_\Urest$. 
 Due to cyclicity, we have $\mTr\in \Ch(U)$.  

Our next theorem states properties of the special symmetric function  on $U$ given by the modified trace.

\begin{theorem}\label{main}
The modified trace $\mTr\in \Ch(\Urest)$ satisfies the following properties.
\begin{itemize}
\item (Partial trace property)
For any $f \in \End_{\Urest}(\Urest^{\otimes 2})$,
\begin{equation*}
\mTr_2\left(f \right)
=\mTr\left( \tr_r^\Urest(f)\right) =\mTr \left( \tr_l^\Urest(f)\right)
\end{equation*}
where $ \tr_l^\Urest$ and $\tr_r^\Urest$ are the right and left partial traces, defined in Equation \eqref{E:PartialLRtrace};

\item (Non-degeneracy)
The pairing \mbox{$\langle\ ,\ \rangle : Z(\Urest) \times \HH(\Urest) \to \C$}
defined by $\langle z,u\rangle= \mTr(z u)$
  is non-degenerate.
\end{itemize}
\end{theorem} 

The proof is by direct computation of the pairing in
the basis of the center and $\HH(U)$, defined in Sections 2 and 4, respectively.

For  any    ribbon Hopf algebra
there exists a 
  {\it universal} invariant
associated to  
an oriented  framed tangle $T$. This invariant
is obtained by assigning the $R$-matrix to crossings and 
evaluation and coevaluation maps to maxima and minima (see Section \ref{S:Universal}).
Although our restricted quantum group $\Urest$ is not ribbon, it has a ribbon extension $\UD$, which produces a universal 
invariant $J_T\in\UD^{\otimes m}$ for any tangle $T$  with $m$ components.
If $T$ is  a string link, we argue that $J_T$ actually belongs
to the subspace $(\Urest^{\otimes m})^\Urest\subset \Urest^{\otimes m}$ of invariants under left action. 

 We are now ready to state our main result.
 Let $\mu\in U^*$ be the right integral of $U$. It is unique up to a normalisation which we fix in Section \ref{S:ResQuGr}.
 For $m\geq 1$,  the function $\mTr_m=t_{\Urest^{\otimes m}}$ defines a bilinear pairing 
 $\langle\ , \rangle:(\Urest^{ \otimes m})^{\Urest} \times \Urest^{ \otimes m} \rightarrow \C$  as follows $$\langle z,x\rangle=\mTr_m(l_zr_x)\ .$$
Here $l_z$, $r_x$ are the left and right multiplications, respectively.
 This bilinear pairing factorises on the right through $\HH(\Urest)^{\otimes m}$.
 
Assume that a closed 3-manifold $M$ 
with a  $(m_+,m_-)$ component framed link $(L^+,L^-)$  inside is represented  
by surgery in $S^3$ along the $m_0$ component link $L^0$. 
Let us color the components of $L^+$ (resp. $L^-$) by central  elements $z_j\in Z(U)$, $1\leq j\leq m_+$ (resp. by trace classes $h_k
\in \HH(U)$, $1\leq k\leq m_-$).
Let $T={T^+ \cup T_0\cup T^-}$ be a string link in $S^3$ 
obtained by opening all $m=m_++m_0+m_-$ components. 
Let $s$ be the signature of the linking matrix for $L^0$
and
$\delta:=\frac{1-i}{\sqrt{2}}q^{\frac{3-\rt^2}{2}}$.
We set $z^+=\otimes_j z_j$ (resp. $h^-=\otimes_k h_k$), 
and  denote by $L$  the colored link $((L^+,z^+), (L^-,h^-))$.
Note that to define $z^+$, $h^-$ as well as $J_T$, we need to fix
an order on the components of  $L$, change of this order will result in an obvious permutation of the entries. 
 
 We define
a number 
associated to the pair $(M, L)$
as follows.
\begin{theorem}\label{E:HenningsLog} With the notation as above,
\begin{equation*}\label{Hlog}
\Hlog(M, L):=\delta^s
\langle\left(   z^+\mu^{\otimes m_+}\otimes \mu^{\otimes m_0}\otimes \Id\right)\left(J_T\right),\,h^-\,
\rangle
\end{equation*}
is a topological invariant of the pair
$(M,  L)$.
\end{theorem}

When $L^-$ is empty,  the colored Hennings invariants \cite{Hennings} are recovered.
When $L^-$ is a knot, then our invariants are equivalent to the  Murakami center valued logarithmic invariants \cite{jM2013}.
 Thus  $\Hlog$ can be understood as a colored extension of the  Murakami invariants.
An action of the modular group $SL(2,\Z)$ on the center $Z(\Urest)$ was studied in \cite{FGST1}. We expect 
that $\Hlog$ can be used  
  to  extend these mapping class group representation in genus one  to a refined  TQFT  with full
 functorial and monoidal properties.

The paper is organized as follows. In Section 2  we 
define the restricted quantum group $\Urest$ and its braided extension $\UD$. In Section 3 we discuss 
their categories of finite dimensional modules. The universal  tangle invariant is constructed in Section 4, where  we also compute a basis
for the space of trace classes $\HH(\Urest)$. Our main theorems are proved in the two last sections.

\subsection*{Acknowledgments}
The authors are grateful to  Azat Gainutdinov and Thomas Kerler for helpful discussions. The work was  supported by
 NCCR SwissMAP founded by the Swiss National Science Foundation. NG was partially supported by NSF grants
DMS-1308196 and MS-1452093. He also would like to thank CNRS for its support.

\section{Restricted quantum $\SL$ and its braided extension}\label{S:ResQuGr}
\subsection*{Definition of $\Urest$}
Fix an  integer $\rt\geq 2$ and let $q=e^\frac{\pi\sqrt{-1}}{\rt}$ be a
$2\rt^{th}$-root of unity.  
  Let $\Urest={\overline\Urest}_q(\slt)$ be the 
$\C$-algebra given by generators $E, F, K, K^{-1}$ and
relations:
\begin{align*}
E^\rt=F^\rt&=0, & K^{2\rt}&=1,
& KK^{-1}&=K^{-1}K=1,\\  KEK^{-1}&=q^2E, & KFK^{-1}&=q^{-2}F, &
  [E,F]&=\frac{K-K^{-1}}{q-q^{-1}}.
\end{align*}
The algebra $\Urest$ is a Hopf algebra where the coproduct, counit and
antipode are defined by
\begin{align*}
  \Delta(E)&= 1\otimes E + E\otimes K, 
  &\varepsilon(E)&= 0, 
  &S(E)&=-EK^{-1}, 
  \\
  \Delta(F)&=K^{-1} \otimes F + F\otimes 1,  
  &\varepsilon(F)&=0,& S(F)&=-KF,
    \\
   \Delta(K)&=K\otimes K
  &\varepsilon(K)&=1,
  & S(K)&=K^{-1},
    \\
  \Delta(K^{-1})&=K^{-1}\otimes K^{-1}
  &\varepsilon(K^{-1})&=1,
  & S(K^{-1})&=K.
\end{align*}
In what follows we will use Sweedler notation. For
$x\in U$ we write
$$\Delta(x)=\sum x_{(1)}\otimes x_{(2)},\quad
\Delta^{[n]}(x)=\sum x_{(1)}\otimes x_{(2)}\otimes ...\otimes
x_{(n)} \quad {\text {for}}\quad n\geq 1.$$

\subsection*{The center of $\Urest$}
The dimension of the center  $Z(\Urest)$ is $3\rt - 1$.
A  basis   consists of
 $\rt+1$ central idempotents $\ideme_j$ $(0\leq j \leq \rt)$ and $2\rt-2$ elements $\idemw_j^\pm$ $(1\leq j\leq \rt-1)$ in the radical \cite{FGST1}.
 These elements satisfy the following  relations:
\begin{align*}
&\ideme_s  \ideme_t =\delta_{s,t} \ideme_s & 0\leq s,t \leq \rt\\
&\ideme_s  \idemw_t^\pm=\delta_{s,t} \idemw_t^\pm & 0\leq s \leq \rt, \; 1\leq t \leq \rt-1\\
& \idemw_s^\pm  \idemw_t^\pm=\idemw_s^\pm  \idemw_t^\mp =0& 1\leq s , t\leq \rt-1.
\end{align*}

\subsection*{Braided extension}
The Hopf algebra $\Urest$ is not braided, see \cite{KondoSaito}.
However, it can be realized as a Hopf subalgebra of the following braided Hopf algebra.  Let $\UD$ be the Hopf algebra generated by $e, \phi, k$ and $k^{-1}$ with the relations:
\begin{align*}
&&e^\rt=\phi^\rt&=0, & k^{4\rt}&=1,\\
 kk^{-1}&=k^{-1}k=1, & kek^{-1}&=qe, & k\phi k^{-1}&=q^{-1}\phi, &
  [e,\phi]&=\frac{k^2-k^{-2}}{q-q^{-1}},
\end{align*}
and Hopf algebra structure:
\begin{align*}
  \Delta(e)&= 1\otimes e + e\otimes k^2, 
  &\varepsilon(e)&= 0, 
  &S(e)&=-ek^{-2}, 
  \\
  \Delta(\phi)&=k^{-2} \otimes \phi + \phi\otimes 1,  
  &\varepsilon(\phi)&=0,& S(\phi)&=-k^2\phi,
    \\
   \Delta(k)&=k\otimes k
  &\varepsilon(k)&=1,
  & S(k)&=k^{-1},
    \\
  \Delta(k^{-1})&=k^{-1}\otimes k^{-1}
  &\varepsilon(k^{-1})&=1,
  & S(k^{-1})&=k.
\end{align*}
The Hopf algebra $D$ has two special invertible elements: the $R$-matrix
$$R=\frac{1}{4\rt} \sum_{m=0}^{\rt -1} \sum_{n,j=0}^{4\rt-1} \frac{(q-q^{-1})^m}{[m]!}q^{m(m-1)/2+m(n-j)-nj/2}e^mk^n \otimes \phi^mk^j$$
 and the ribbon element 
$$\qv=\frac{1-i}{2\sqrt{\rt}}  \sum_{m=0}^{\rt -1} \sum_{j=0}^{2\rt-1} \frac{(q-q^{-1})^m}{[m]!} q^{-m/2 +mj + (j+p+1)^2/2}\phi^me^mk^{2j}$$
where $q^{\frac{1}{2}}=e^{\frac{i\pi}{2\rt}}$, $[m]=\frac{q^m-q^{-m}}{q-q^{-1}}$ and $[m]!=[m][m-1]...[1]$.
The following theorem is well known, see \cite{FGST1}.
\begin{theorem}\label{T:UDbraided}
The triple $(\UD, R,r)$ is a ribbon 
 Hopf algebra.
\end{theorem}



Let  us call
$M=R_{21} R$  the double braiding or {\it monodromy},
where  $R_{21}=\sum_i \beta_i\otimes \alpha_i$
with  $R=\sum_i \alpha_i\otimes \beta_i$.
A Hopf algebra $A$ is called {\it factorisable} if its monodromy matrix
can be written as
$$M=\sum_i m_i\otimes n_i$$
where $m_i$ and $n_i$ are two bases of $A$.
The Hopf algebra $D$ is not factorisable.
There is a Hopf algebra embedding $\Urest\to \UD$ given by 
$$ E\mapsto e, \; F\mapsto \phi,\; \text{ and } K \mapsto k^2.$$ 
It is easy to check that
$\qv\in \Urest$,  and the 
monodromy
$M= R_{21} R \in \Urest\otimes \Urest$.
Moreover, $\Urest$ is factorisable.
\subsection*{Ribbon and balancing elements of $U$}
 Let $u=\sum_i S(\beta_i)\alpha_i$ 
   be the canonical element implementing the inner-automorphism $S^2$,
  i.e. $S^2(x)=u x u^{-1}$ for any $x\in \UD$,
  and satisfying
  $$\Delta(u)=M^{-1}(u \otimes u).$$
 Using the formula for the $R$-matrix, 
 it is easy to check that $u\in \U$.
The ribbon element $r\in \U$ is central and invertible, 
such that
\begin{equation}\label{ribbon}
r^2=uS(u), \quad S(r)=r,   \quad
\varepsilon(r)=1 \quad \Delta(r)=M^{-1}(r\otimes r).
\end{equation} 
Using them we define
 $$
g:=\qv^{-1}u=k^{2\rt +2}=K^{\rt +1}\in \U$$ 
  the {\it balancing}   element. 
 This element is grouplike, i.e.
 \begin{equation}
 \label{bal}
 \Delta(g)=g\otimes g,
  \quad \varepsilon(g)=1,\quad {\text{ and}} \quad
 gxg^{-1}=S^2(x) \end{equation}
  for any  $x\in \U$.
 The balancing  element
 will be used to define the pivotal structure. 
 
 \begin{remark} As it was shown by Drinfeld,
 equations \eqref{bal}
 determine $g^2$ only.
 Hence $g'=K$ is another balancing element, which  will not be considered in this paper.
\end{remark}


\subsection*{Right integral}
Recall that a {\em right integral}  $\mu\in U^*$ is defined
by the following system of equations
$$(\mu\otimes \Id)\Delta(x)=\mu(x) \one
\quad{\text{for all}}\quad x\in \Urest .
$$ 
For any finite dimensional Hopf algebra over
a field of zero characteristics, there is a unique solution to these equations up to a scalar.
In our case in the PBW basis, it is given by the formula 
$$\mu(E^mF^nK^l)=\zeta \delta_{m,\rt-1}\delta_{n,\rt-1}\delta_{l,\rt+1}\ .$$
We fix  normalisation as in \cite{jM2013} by setting
 $$\zeta=-\sqrt{\frac{2}{\rt}}([\rt-1]!)^2
 .$$
The evaluation on the ribbon element and its inverse are given by 
$$\mu(r) =\frac{1-i}{\sqrt{2}}q^{\frac{3-\rt^2}{2}}
=\frac{1}{\mu(r^{-1})}=\delta\ .$$
One can show that $\mu$ belongs to the space of quantum characters
$$\qCh(\Urest):=\{\phi \in \Urest^* \;|\; \phi ( xy)=\phi(S^2(y)x) \quad\text{for any}\quad x,y\in \Urest\} .$$
 The center $Z(\Urest)$ acts on
 $\qCh(\Urest)$ by $z\phi(x):=\phi(zx)$ for any $z\in Z(U)$ and $x\in U$. Under this action
   $\qCh(\Urest)$ is a free module of dimension one with basis given by the right integral $\mu$.
 Hence, as a vector space $\qCh(U)$ has dimension $3p-1$.

The space of quantum characters $\qCh(\Urest)$ is naturally isomorphic to the
space of characters.
%
Indeed, we can define the map
\begin{equation}\label{ch-qch}
Q:\qCh(U) \to \Ch(U) \quad{\text {by sending}}\quad
\phi \mapsto \phi_g
\end{equation}
where $\phi_g(x):=\phi(g x)$ and $g$ is
the balancing element. Cyclicity can be verified as follows:
\begin{align*}\label{qChar}
\phi_g(xy)=\phi(g xy)=\phi ( S^2(y) g x)=
\phi(gyx)=\phi_g(yx) 
\end{align*}
The inverse map is given by sending
 $\psi \in \Ch(A)$ to  $\psi_{g^{-1}}
\in \qCh(A)$. Hence $Q$ is an isomorphism.  
%

We get that the dimension of the dual vector spaces $\Ch(\Urest)$ and $\HH(\Urest)$ is also $3p-1$.

\section{Categories of modules }\label{S:modules}
In this section we will study 
 the $\C$-linear categories of finite dimensional modules over $\UD$ and
 $\Urest$, which we denote by $\Dmod$ and $\Umod$, respectively.

\subsection*{Category $\Dmod$}
The category $\Dmod$ is ribbon, with the usual braiding
$$c_{V,W}: V\otimes W\to W\otimes V,\quad \text{ given by }\quad u\otimes v \mapsto \tau R(u\otimes v),$$
 where $\tau(x\otimes  y)=y\otimes x$,   twist 
\begin{equation*}\label{E:twisttheta}
\theta_V:V\to V,\quad\text{given by}\quad  v\mapsto \qv^{-1} v
\end{equation*}
and compatible duality
\begin{align}\label{E:DualitiesC}
&\coev_{V} : \:\: \C \rightarrow V\otimes V^{*} , \text{ given by } 1 \mapsto \sum_i v_i\otimes v_i^*, \notag\\
&\ev_{V}: \:\: V^*\otimes V\rightarrow \C, \text{ given by } f\otimes v \mapsto f(v), \notag\\ 
&\tcoev_V: \:\: \C \rightarrow V^{*}\otimes V\quad \text{ given by } 
1 \mapsto \sum_i v_i^*\otimes g^{-1}v_i, \notag\\
& \tev_V: \:\: V\otimes V^*\rightarrow \C \text{ given by } v\otimes f \mapsto f(gv)
\end{align}
where $g$
is the balancing element. Using properties of $r$, one can check that
the twist is {\it self-dual}, i.e.
$\theta_{V^*}=\theta^*_V$
.

The duality morphisms \eqref{E:DualitiesC} define 
pivotal structure on $\Dmod$ (see e.g. \cite{GPV}). In particular, 
in the pivotal setting, 
 one can define left and right
(categorical) traces of any endomorphism $f: V \to V$ as
$$\tr_l(f)=\ev_V(\Id_{V^*}\otimes f)\tcoev_V\quad
\tr_r(f)=\tev_V(f\otimes \Id_{V^*})\coev_V \, $$ and dimensions of objects.

A {\it spherical} category is a pivotal category whose
left and right traces are equal, i.e.
$\tr_l(f)=\tr_r(f)$ for any endomorphism $f$. 
It is easy to see
that any ribbon category is spherical. We call
$\tr_l(\Id_V)=\tr_r(\Id_V)$ the {\it quantum} dimension of $V$.

We will use a standard graphical calculus to represent 
morphisms in $\Dmod$   by diagrams in the plane which are read from the bottom to the top.  

In what follows, we will need {\it partial} categorical traces
of endomorphisms.
Given $V,W\in \Dmod$ and $f:V\otimes W\to V\otimes W$ let $\tr^V_l$ and $\tr^W_r$ be the left and right {\it partial} traces defined as follows
\begin{align}
&\tr_l^V(f)=(\ev_V \otimes \Id_W)(\Id_{V^*} \otimes f)(\tcoev_V \otimes \Id_W)= \put(14,20){{\tiny $V$}} \epsh{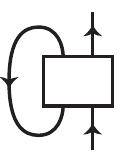}{45pt} \put(-19,-1){{\footnotesize $f$}}\put(-11,20){{\tiny $W$}} \notag \\
&\tr_r^W(f)=(\Id_V \otimes \tev_W)(f \otimes \Id_{W^*})(\Id_V \otimes \coev_W)= \put(17,-18){{\tiny $W$}}\epsh{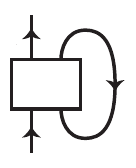}{45pt}  \put(-27,0){{\footnotesize $f$}}\put(-37,-18){{\tiny $V$}}. \label{E:PartialLRtrace}
\end{align}

\subsection*{Category $\Umod$}
\label{SS:DualityPivStr}
Let us call a module   {\it simple}, if
its endomorphism ring is one dimensional. A module
is {\it projective} if it is a direct summand of a free module.

The category $\Umod$ includes
 the $s$-dimensional simple modules $ \repX^\pm_s$, 
 and their  projective covers $\PP^\pm_s$
 for $1\leq s\leq \rt$, which are
 $2\rt$ dimensional for $1\leq s< \rt$.
The simple module $ \repX^\pm_s$ is
determined by its highest weight vector $v$ with the action $Ev=0$ and $Kv=\pm q^{s-1}v$. 
It  is projective if and only if $s=p$.

A category  is called {\it unimodular}, if the projective cover of the trivial module is self-dual.  Since
$\PP^+_1$ is self-dual (as well as all other modules defined above), $\Umod$  is unimodular.

The category $\Umod$ 
inherits the pivotal structure,
twist and  double braiding  from $\Dmod$.
 The double braiding is
$$M_{V,W}: V\otimes W\to V\otimes W,\quad \text{ given by }\quad x\otimes y \mapsto M(x\otimes y),$$
where $M$ is the monodromy matrix;
the self-dual  twist  and duality are given by
\eqref{E:DualitiesC}.
It can be checked that $\Umod$ is {\it twisted} category 
with duality in the sense of Brugui\`eres \cite{Br}.


Let $\Upmod$ be the full  subcategory of $\Umod$ consisting of projective modules. This category is non-abelian. 
To compute the modified trace on $\Upmod$, we will need an
explicit structure of this category.

\subsection*{Structure of $\Upmod$}
A module is {\it indecomposable} if it does not
decompose as a direct sum of two
modules.
 The indecomposable projective $\Urest$-modules are classified up to isomorphism in \cite{FGST2}: they are precisely the
 projective covers $\PP^\pm_j$ of the simple modules  
 where  $j=1,...,\rt$.   In particular, $\PP^\pm_\rt$ is a simple module with highest weight $\pm q^{\rt-1}$. The module $\PP^+_1$ is the projective cover of the trivial one.  

We will recall some facts about these projective modules. For $1\leq j\leq \rt-1$ let   
\begin{equation}\label{E:BasisP}
\{\junxpm_k,\junypm_k\}_{0\leq k \leq \rt-j-1}\cup \{\junapm_n,\junbpm_n\}_{0\leq n \leq j-1}
\end{equation}
be the basis of $\PP^\pm_j$ given in \cite{FGST1} (see Section C.2 of \cite{FGST1} for the defining relations).  

Following \cite{CGP3} we call a weight vector $v$ \emph{dominant} if $(FE)^2v = 0$. The vector $\junbp_0$ (resp. $\junym_0$) is a dominant  vector of $\PP^+_j$ (resp. $\PP^-_j$) with weight $\pm q^{j-1}$.  Let $x_j^+$ (resp. $x_j^-$) be the nilpotent endomorphism of $\PP^\pm_j$ determined by $\junbp_0\mapsto \junap_0$ (resp. $\junym_0\mapsto \junxm_0$), and let  $a_j^+, b_j^+: \Proj_j^+\rightarrow \Proj_{\rt-j}^{-}$ and  $a_j^-, b_j^-: \Proj_{j}^{-}\to \Proj_{\rt-j}^+$ be the morphisms defined by 
 $$a_j^+(\junbp_0)= \junam_0, \;\; b_j^+(\junbp_0)= \junbm_0, \;\; a_j^-(\junym_0)= \junxp_0 \text{ and } b_j^-(\junym_0)= \junyp_0,$$
  respectively. 
  Analysing the images of the dominant weight vector of $\Proj_s^\epsilon$, we  can completely determine the $\Hom$-spaces between indecomposable projective modules.  Here is the list 
  of the non-trivial ones:
  \begin{itemize}
  \item    
  the endomorphism ring $\End_\Urest(\PP^\pm_j)$ is one dimensional  for $j=p$ and  two dimensional with basis $\{\Id_{\PP^\pm_j},x_j^{\pm}\}$
 for $1\leq j <\rt$,
 \item the $\Hom$-spaces $ \Hom_\Urest(\Proj_j^+,\Proj_{\rt-j}^{-})$ and $ \Hom_\Urest(\Proj_{j}^-,\Proj_{\rt-j}^{+})$)
 are two dimensional with respective basis $\{a_j^+, b_j^+\}$ and $\{a_j^-, b_j^-\}$, for $1\leq j <\rt$.
\end{itemize}

\begin{proposition}[Proposition 4.4.4 of \cite{FGST1}]\label{P:Prop444}
  The action of the center  on the indecomposable 
  projective modules  is as follows.
$$
\begin{array}{|c||c|c|c|c|}
\hline
&\Proj_p^-&\Proj_p^+&\Proj_j^+,\text{ $1\leq j<p$}& \Proj_{p-j}^-,\text{ $1\leq j<p$}\\
\hline\hline
\ideme_0&\Id_{\Proj_p^-}&0&0&0
\\ \hline
\ideme_p&0&\Id_{\Proj_p^+}  &0&0
\\ \hline
\ideme_j,\text{ $1\leq j<p$} &0&0&\Id_{\Proj_j^+}&\Id_{\Proj_{p-j}^-} \\ \hline
\idemw_j^+,\text{ $1\leq j<p$}&0&0&x_j^+&0 \\ \hline
\idemw_j^-,\text{ $1\leq j<p$}&0&0&0&x_{p-j}^-\\ \hline
\end{array}$$ 
 
\end{proposition}
\vspace*{2mm}

 \section{Tangle invariants  and
 the trace of $\U$}\label{S:Universal}
Our links and tangles are always assumed to be framed and oriented. The diagrams are going from bottom to top.
A {\it string link} is a tangle without closed component whose arcs  end
at the same order as they start, with upwards orientation. A pure braid is an example.

\subsection*{Reshetikhin--Turaev invariant}
Given any ribbon category $\mathcal C$,
Turaev showed in \cite{Tu}, that there exists a 
canonical ribbon functor 
$$F_{\mathcal C}:\Rib_{\mathcal{C}}\to \mathcal C,$$
where $\Rib_{\mathcal{C}} $ is the category of $\mathcal{C}$-colored ribbon graphs.
Applying this construction to  $\Dmod$, for an $m$-component string link $T$, colored with the regular representation $D$, 
we obtain $ F_{\UD}(T)\in \End_\UD(\UD^{\otimes m})$.
Here we use shorthand $F_{\UD}$ for $ F_{\Dmod}$. 


The  Reshetikhin--Turaev invariant
of the colored 
link $(L, V_{1}, \dots, V_{m})$
is obtained by evaluating
the categorical right traces on $F_D(T)$, i.e.
\begin{align*}\label{closure}
J_L(V_1, \dots,V_m):&=(\tr^{V_1}_r \otimes \dots \tr^{V_m}_r) F_D(T)\\ &=
({\rm trace}^{V_1}\otimes\dots \otimes {\rm trace}^{V_m})(g\otimes \dots \otimes g)
F_D(T) 
\end{align*}
where  $T$ is a string link with braid closure isotopic to  $L$. 
Recall also that $\tr^V_r=\tr^V_l$ since both closures are isotopic. Hence, we can replace $g$ by $g^{-1}$ in the last line.
If
 one of the $V_i$'s is projective, this invariant vanishes.
\subsection*{Universal invariant}
Associated with a 
ribbon Hopf algebra there is another 
powerful invariant -- the {\it universal} invariant of links and tangles introduced by Lawrence \cite{Lawrence2} for some quantum groups 
and defined by Hennings \cite{Hennings} in the general case.
For  a string link $T$ with $m$ components, its universal invariant $J_T$ is obtained by pasting together
pieces shown in  Figure \ref{F:univ_picture}.
Here we write $R=\sum \alpha \otimes \beta$, $R^{-1}=\sum \overline \alpha \otimes \overline\beta$.
\begin{figure}[h]
 \centering
 \begin{tabular}{cccc}
 \Png{1.1cm}{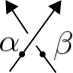}\ &\ \Png{1.1cm}{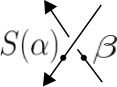}\  &\ \Png{1.1cm}{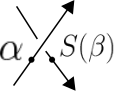}\  &\ \Png{1.1cm}{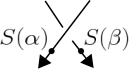}\\[12pt]
 \Png{1.1cm}{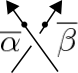}\ & \  \Png{1.1cm}{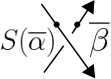}\ & \  \Png{1.1cm}{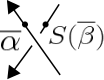}\  &\ \Png{1.1cm}{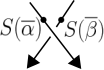}\\[12pt]
  \Png{0.7cm}{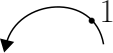} &\Png{0.73cm}{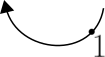}&
  \Png{0.7cm}{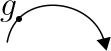} & \Png{0.65cm}{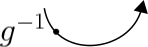}
  \end{tabular}
  \caption{\label{F:univ_picture} Local formulas for the universal invariant}
 \end{figure}

More precisely, for each arc of $T$ we obtain an element of $\UD$ by writing a word from right to left with labels read using the order
following the orientation. Thus we get $J_T\in \UD^{\otimes m}$. This element does not change by Reidemeister moves.
The original proof of the invariance was stated for a link. The argument was extended to tangles  in  \cite[7.3]{Hab_bottom}.

Note that in \cite{Hab_bottom} Habiro uses different 
conventions. His tangles are depicted from top to bottom and orientations are reversed. Hence,
our model can be recovered from his one after reflecting  over a horizontal axis. The universal invariants coincide.
In \cite{Oht}, Ohtsuki defines the universal invariant using  the opposite to our orientation convention, but  the word is written there
from left to right when following the orientation. Again the universal invariant is finally the same as ours.


\subsection*{Relation between them}
The universal invariant is known to dominate Reshetikhin--Turaev invariants in the following sense:
$$J_L(V_1, \dots, V_m)=({\tr}^{V_1}_r \otimes \dots \otimes {\tr}^{V_m}_r) J_T$$
 The proof is given in
 \cite[Theorem 4.9]{Oht}.

Let us denote by $(\Urest^{\otimes m})^{\Urest}\subset \Urest^{\otimes m}$ the submodule centralising the left action, i.e.
 $$x\in  (\Urest^{\otimes m})^{\Urest}\quad
 \Leftrightarrow \quad\Delta^{[m]}(h) x= x \Delta^{[m]}(h) \quad{\text {for all}} \quad h\in \Urest\ .$$

The following lemma is folklore, but we are adding the proof for completeness.

\begin{lemma}\label{E:RTleftAction}
The Reshetikhin--Turaev intertwinner $ F_{\UD}(T)$
is equal to left multiplication by $J_T$. In addition, 
for an $m$-component string link $T$, $J_T$
belongs to $(\Urest^{\otimes m})^{\Urest}$.
\end{lemma}
\begin{proof}
The fact that $F_D(T)$ is the left multiplication by $J_T$
follows directly by comparing the definitions of these
two invariants. More details are given in
the proof of \cite[Theorem 4.9]{Oht}. Hence, 
 multiplication by $J_T$ has to commute with left action,
 we conclude $J_T\in (D^{\otimes m})^D$.

Let us show that for a string link the universal invariant $J_T$ actually belongs to $\Urest^{\otimes m}$. This implies the claim, since $ (\Urest^{\otimes m})^D \subset (\Urest^{\otimes m})^U$.


The linking matrix of a string link is diagonal mod $2$.
In \cite[Section 1.3]{Br}, Brugui\`eres shows that
 {\it any} tangle with this property can be obtained as compositions and tensor products of evaluations, coevaluations and twists.
Thus  the universal invariant $J_T$ is build up from the ribbon element $r$, 
the balancing element $g$,  and their inverses
by applying the Hopf algebra operations.
We obtain $J_T\in \Urest^{\otimes m}$.
  \end{proof}

\subsection*{Evaluations of the universal invariant}
Assume for simplicity that $T$ is a $(1,1)$ tangle,
whose closure is the knot $K$.
Then for any $\phi\in \qCh(U)$, the evaluation
$\phi(J_T) \in \C$ is a knot invariant. To prove this
fact, we need to show that this evaluation does not change by cyclic permutations of the word $J_K= g^{-1}J_T$ obtained
 by applying the  algorithm described above to the 
left closure of $T$. 
Using  \eqref{ch-qch}, we get
$$\phi( J_T)=\phi_g(g^{-1} J_T)= \phi_g (J_K).$$
The last expression does not change by cyclic permutations
since $\phi_g\in \Ch(U)$.
Applying this argument to the $k$ leftmost components of a string link with $m$ strands, we will
get the following.
\begin{lemma}\label{XXX}
Let $T$ be an $m$-component string link and $1\leq k\leq m$. Let $\Phi=\otimes_{j=1}^k \phi_j$ with
$\phi_j\in \qCh(U)$ be a sequence of quantum characters, then
$$ (\Phi \otimes \Id_{U^{\otimes(m-k)}})(J_T)\in (U^{\otimes (m-k)})^U$$
is  invariant of the tangle  obtained from $T$ by closing the  first $k$ components.
\end{lemma}

Further observe that given $\phi\in \qCh(U)$ 
we can obtain another quantum character $\phi_z$ by twisting
$\phi$ with a central element $z\in Z(U)$, where
$\phi_z(x):=\phi(zx)$ for any $x\in U$.
In the logaritmic invariant
 we will use  quantum characters
$\mu$ and $\mu_z$  to evaluate components of $T^0$ and $T^+$,
and the {\it modified trace pairing}
constructed  below for $T^-$.

In the next section we will
define  a family of linear functions
$$\{\t_V:\End_\cat(V)\rightarrow \C \}_{V\in \Upmod}$$
satisfying cyclicity and the partial trace property.
For any $m>0$,   we can use $\t_{U^{\otimes m}}$
 to define 
 a bilinear pairing
$$\langle\ ,\ \rangle:\left(\Urest^{\otimes m}\right)^{\Urest}\otimes 
\Urest^{\otimes m}\rightarrow \C\ ,$$
by the formula
\begin{equation}\label{E:pairing}
\langle x ,y\rangle=\t_{U^{\otimes m}}(l_x\circ r_y)\ .
\end{equation}
Here $l_x$, $r_x$ are the operators of the left and right
multiplication by $x$. 

From cyclicity of $\t$, we obtain an induced
pairing
$$\langle\ ,\  \rangle:\left(\Urest^{\otimes m}\right)^{\Urest}\otimes 
\HH(\Urest^{\otimes m})\rightarrow \C\ ,$$
which  we call the modified trace pairing. 
To achieve the full evaluation,  besides of the basis
for the center, we will need a basis for $\HH(U)$.

\subsection*{The trace $\HH(\Urest)$}
Let us construct a basis of the trace of $U$.

Recall that 
 $0^{\text{th}}$-Hochschild homology or trace of a 
 linear category $\mathcal C$
 is defined by
$$\HH(\mathcal C):=\frac{\oplus_{x\in \mathcal C} \;\mathcal C(x,x)}{fg-gf}\quad
\text{for any}\quad{f:x\to y,g:y\to x }.$$
The image of $x\in \mathcal C$ in $\HH(\mathcal C)$ will be called its trace class and denoted by $[x]$.
For an algebra $A$ (i.e. a category with one object)
this reduces to
$$ \HH(A):=\frac{A}{[A,A]}\quad\text{with}\quad
[A,A]={\rm{Span}}\{xy-yx\;|\;  x,y\in A\}.$$
This space supports a natural action of the center defined by
$z[x]=[zx]$ for any $z\in Z(U)$.
By definition $\HH(\Urest)$ is dual to $\Ch(\Urest)$.

We will use the following well-known fact.
 \begin{proposition} \label{Meq}
 For any finite dimensional algebra $U$, 
$$\HH ({\rm \Upmod})\simeq\HH(U)$$
\end{proposition}
\noindent
 Let us give a proof for completeness.
\begin{proof}
 Assume $M$ is projective, then there exists another
projective module $N$ such that 
$M\oplus N = U^{\otimes n}$. Hence in $\Upmod$ there are morphisms
 $$i:M\to U^{\otimes n}\quad\text{and}\quad p:U^{\otimes n}\to M$$
 with $p \circ i=\Id_M$ and $i\circ p$ an idempotent.
 They can be used to define a map $\HH(\Upmod)\to\HH(U)$
 as follows: Given $f\in\End(M)$, then
 $$[f]=[fpi]=[ifp] \in \HH(U^{\otimes n})\simeq\HH(U) .$$
 The last equivalence in proven in \cite{Loday}.
The inverse of this map sends $x\in \HH(U)$ to the right multiplication $r_x$ in $\End_U(\Urest)$. Hence, we have an isomorphism.
\end{proof}

Let us recall that $U$ as a free left $U$-module decomposes
into a direct sum of indecomposable projectives as follows
$$\Urest \cong \bigoplus_{j=1}^{\rt} j \PP_j^+ \oplus 
\bigoplus_{j=1}^{\rt} j \PP_j^-.
$$
The module $ \mathbb{P}=\oplus_{j,\pm}\ \Proj_j^\pm$ is called  projective generator of $\Urest$ and 
  $B=\End(\mathbb{P})$ the {\it basic} algebra.
  By definition,
  $$\HH(\Upmod)=\HH(B).$$

\begin{remark}  By \cite[Th. 8.4.5]{DK},  any finite dimensional algebra $A$ and its basic algebra $B$
   are Morita equivalent.  
  The Morita equivalence between $A$ and $B$ also implies
$$ \HH(A\text{{\rm -mod}})\simeq\HH(B\text{\rm{-mod}})$$ but this group could be different from $\HH(A)$.  
\end{remark}
 
 Let us use the notation $\Id_j^\pm=\Id_{\PP^\pm_j}$ for simplicity.
\begin{lemma}\label{L:HH0}
A basis for $\HH(B)$ is represented by $[\Id_k^\pm]$, $1\leq k\leq \rt$, and $[x_j^+]=[x_{\rt-j}^-]$, $1\leq j< \rt$.
\end{lemma}
\begin{proof}

Recall that 
$$B=\bigoplus_{j,\epsilon,j',\epsilon'} \Hom(\Proj_j^\epsilon,\Proj_{j'}^{\epsilon'})\ .$$
A linear basis for $B$ defined in  Section \ref{S:modules} 
 consist of
$$\Id_k^\pm, \quad x^\epsilon_j,\quad a_j^\epsilon: \Proj_j^\epsilon\rightarrow \Proj_{\rt-j}^{-\epsilon},\quad
\text{and}\quad b_j^\epsilon: \Proj_j^\epsilon\rightarrow \Proj_{\rt-j}^{-\epsilon}$$
where $1\leq k\leq \rt$, $1\leq j< \rt$ and $\epsilon\in \{-,+\}$.
For $1\leq j<\rt$,
we have $a_j^\epsilon \Id_j^\epsilon=a_j^\epsilon$,
while $\Id_j^\epsilon a_j^\epsilon=0$, so that $a_j^\epsilon$ and similarly $b_j^\epsilon$ vanish in $\HH(B)$. 
We also have $ b_{\rt-j}^{-\epsilon}a_j^\epsilon=x_j^\epsilon$, while 
 $ a_j^\epsilon b_{\rt-j}^{-\epsilon}=x_{\rt-j}^{-\epsilon}$. We get the relation
 $[x_j^\epsilon]=[x_{\rt-j}^{-\epsilon}]\in \HH(B)$.
 Since the resulting  set of generators has expected cardinality,
this completes the proof of the lemma.
 \end{proof}

Combining Proposition \ref{Meq} with Lemma \ref{L:HH0}, we conclude that
 $\HH(U)$ has  dimension $3p-1$, with basis consisting of
 \begin{itemize}
\item $\tclass^\pm_k$, for $1\leq k \leq \rt$, represented by the minimal (non central) idempotent
projecting  onto a copy of the  module $\PP_k^\pm $, and
\item $\tclass_j=\idemw_j^+\tclass_j^+=\idemw_j^-\tclass_{p-j}^-$, for $1\leq j \leq \rt-1$. 
 \end{itemize}
On the regular representation in $\Upmod$ these elements act by the right multiplication.

\section{Proofs of Theorems \ref{mTrace} and \ref{main}}

In this section we construct the modified trace on $\Upmod$  and  compute it for the regular
representation $\Urest$ and its tensor powers. 
This will provide the main tool
to prove Theorems \ref{mTrace} and \ref{main}.

\subsection*{Modified trace}
The subcategory $\Upmod$ is an ideal of $\Umod$ in the sense of \cite{GKP2, GPV},
 which means  the following:  
\begin{enumerate}[label=\alph*)]
 \item If $V\in \Upmod$ and $W\in \Umod$, then $V\otimes W \in \Upmod$ and $V\otimes W \in \Upmod$.
\item If $V\in \Upmod$, $W\in \Umod$, and there exists morphisms  $f:W\to V$,  $g:V\to W$ such that $g  f=\Id_W$, then $W\in \Upmod$.
\end{enumerate}
Let us recall that a {\em modified trace} on $\Upmod$  is  a family of linear functions
$$\{\t_V:\End_\cat(V)\rightarrow \C \}_{V\in \Upmod}$$
such that the following two conditions hold:
\begin{enumerate}
\item {\bf Cyclicity.}  If $X,V\in \Upmod$, then for any morphisms $f:V\rightarrow X $ and $g:X\rightarrow V$  in $\Urest$-mod we have 
\begin{equation*}\label{E:Tracefggf}
\t_V(g f)=\t_X(f  g).
\end{equation*} 
\item {\bf Partial trace properties.} If $X\in \Upmod$ and $W\in \Umod$ then for any $f\in \End_\Urest(X\otimes W)$ and $g\in \End_\Urest(W\otimes X)$ we have
\begin{align*}%
\t_{X\otimes W}\left(f \right)
&=\t_X \left( \tr_r^W(f)\right),  \\
 \t_{W\otimes X}\left(g \right)
&=\t_X \left( \tr_l^W(g)\right)  
\end{align*}
where $\tr_r^W$ and $\tr_l^W$ are the right and left partial categorical traces along $W$ defined by \eqref{E:PartialLRtrace}.  
\end{enumerate}
If only the first (resp. the second) of the two partial trace properties is satisfied, we call the modified trace
right (resp. left).
\subsection*{Proof of Theorem \ref{mTrace}}
Corollary 3.2.1 of \cite{GKP2} implies the existence of a unique (up to global scalar) right modified trace 
in any unimodular pivotal category with enough projectives and a simple projective object $L$, such that $\tev_L$  is surjective.
The category $\Upmod$ does satisfy all these assumptions.
Hence there exists a unique right  modified trace
$$\{\rTr_V:\End(V)\rightarrow \C \}_{V\in \Upmod}
$$
  normalized by 
$$\rTr_{\PP^+_\rt}(\Id_{\PP^+_\rt})=(-1)^{\rt-1}.$$
Analogous arguments imply the existence of the unique
left trace
$$ \{\lTr_V:\End(V)\rightarrow \C \}_{V\in \Upmod}$$
  normalized by 
$$\lTr_{\PP^+_\rt}(\Id_{\PP^+_\rt})=(-1)^{\rt-1}.$$


We will compute both of them and show that they coincide.
For this, we will use  additivity of trace functions for direct sums,
which follows  from the cyclicity.
Hence, it is enough to compute modified traces on the 
endomorphisms of the indecomposable projectives.

\subsubsection*{Identity endomorphisms}
Recall $\repX_1^+$ is the one dimensional trivial $\Urest$-module whose action on any vector $v$ is given by $Ev=Fv=0$ and $K^{\pm 1}v=v$.  There is another one dimensional module $\repX_1^-$ whose action on any vector $v$ is given by $Ev=Fv=0$ and $K^{\pm 1}v=-v$.  
Hence, $\tr_r(\Id_{\repX_1^-})=\text{trace}^{\repX_1^-}(K^{p+1})=
(-1)^{p+1}$. 
  Using $\PP^-_j\cong \PP^+_j\otimes\repX_1^-$, we can compute the  trace on $\PP^-_\rt$:
 $$\rTr_{\PP^-_\rt}(\Id_{\PP^-_\rt})=\rTr_{\PP^+_\rt\otimes \repX_1^-}(\Id_{\PP^+_\rt\otimes \repX_1^-})=\rTr_{\PP^+_\rt} \left( \tr_r^{\repX_1^-}(\Id_{\PP^+_\rt\otimes \repX_1^-})\right)=1.
 $$
 Similarly, $\PP^-_\rt\cong  \repX_1^- \otimes \PP^+_p $ implies that $\lTr_{\PP^-_\rt}(\Id_{\PP^-_\rt})=1$.  \\

Next we compute some partial traces involving $\repX^+_2$.  
The module $ \repX^+_2$ has a basis $\{w_0,w_1\}$ with action
\begin{align*}
Ew_0 &=0 & Fw_0&=w_1 & Kw_0&=qw_0\\
Ew_0 &=w_1 & Fw_1&=w_0 & Kw_1&=q^{-1}w_1.
\end{align*} 
Recall that  
an endomorphism of $\PP^+_j$ is determined by the image of the dominant vector $\junbp_0$.  We can use this fact to compute the following partial trace:
\begin{multline*}\left(\tr_l^{\repX^+_2}(\Id_{\repX^+_2\otimes \PP^+_j})\right)(\junbp_0)=\\
=w_0^*( K^{\rt-1}w_0) \junbp_0 + w_1^*( K^{\rt-1}w_1) \junbp_0 =(q^{\rt-1}+q^{-\rt+1})\junbp_0\ .
\end{multline*}
Thus, $\tr_l^{\repX^+_2}(\Id_{\repX^+_2\otimes \PP^+_j})=-(q+q^{-1})\Id_{ \PP^+_j}$ for $j=1,...,\rt-1$.
Similarly, for the right partial trace of the identity of $ \PP^+_j\otimes\repX^+_2$ 
we get
$$
\tr_r^{\repX^+_2}(\Id_{ \PP^+_j\otimes\repX^+_2})=(q^{-\rt+1}+q^{\rt-1})\Id_{ \PP^+_j}=(-q-q^{-1})\Id_{ \PP^+_j}.
$$

The decomposition of tensor products of a simple module with a projective indecomposable module is given in \cite[proposition 4.1]{Suter} (also see \cite[Theorems 3.1.5, 3.2.1]{KondoSaito}).
In particular, 
\begin{equation}\label{E:IsoPPp} 
\PP_\rt^\pm\otimes \repX_2^+\cong \PP^\pm_{p-1}\cong \repX_2^+\otimes \PP_\rt^\pm\ ,
\end{equation}
\begin{equation}\label{E:IsoPP1} 
\PP_{\rt-1}^\pm\otimes \repX_2^+\cong \PP_{\rt-2}^\pm\oplus 2\PP^\pm_\rt\cong \repX_2^+\otimes \PP^\pm_{\rt-1}\ ,
\end{equation}
and 
\begin{equation}\label{E:IsoPPj} 
\PP^\pm_j\otimes \repX^+_{2}\cong \PP^\pm_{j-1}\oplus \PP^\pm_{j+1}\cong \repX^+_{2}\otimes \PP^\pm_j
\end{equation} 
for $j\in\{2,...,\rt-2\}$. 

Combining  these formulas with properties of the modified trace we have:
$$\rTr_{\PP^\pm_j\otimes \repX^+_{2}}\left(\Id_{ \PP^+_j\otimes\repX^+_2}\right)
=\rTr_{\PP^\pm_j}\left(\tr_r^{\repX_2^+}(\Id_{ \PP^+_j\otimes\repX^+_2})\right)=
(-q-q^{-1})\rTr_{\PP^\pm_j}\left(\Id_{ \PP^+_j}\right)$$
where $j\in \{2,...,\rt\}$.  
This equality  together with the isomorphism on the left hand side of  Equation \eqref{E:IsoPPp} for $j=\rt$ gives
$$ \rTr_{\PP^\pm_{\rt-1}}(\Id_{\PP^\pm_{\rt-1}})=(\mp 1)^{\rt-1}(-q-q^{-1}).$$
Then the isomorphism of the left hand side of \eqref{E:IsoPP1} implies
$$ \rTr_{\PP^\pm_{\rt-2}}(\Id_{\PP^\pm_{\rt-2}})=(\mp 1)^{\rt-1}((-q-q^{-1})^2-2)=(\mp 1)^{\rt-1}(q^2+q^{-2}).$$
Finally, for $j\in\{2,...,\rt-2\}$, the isomorphism of the left hand side of Equation \eqref{E:IsoPPj} implies
\begin{equation}\label{E:rTrPPpmchi}
-(q+q^{-1})\rTr_{\PP_j^\pm}(\Id_{\PP_j^\pm})=\rTr_{\PP_{j-1}^\pm}(\Id_{\PP_{j-1}^\pm})+\rTr_{\PP_{j+1}^\pm}(\Id_{\PP_{j+1}^\pm})\ .
\end{equation}
Recursively the last equality implies 
\begin{equation}
\rTr_{\PP_j^\pm}(\Id_{\PP_j^\pm})=(\pm 1)^{p-1}(-1)^{j}(q^j+q^{-j})
\end{equation}
for  $j\in\{2,...,\rt-1\}$.  

Using the right hand side of the tensor product in \eqref{E:IsoPPp},\eqref{E:IsoPP1},\eqref{E:IsoPPj} we compute similarly the
left trace of identities and get
$$\lTr_{\PP^\pm_j}(\Id_{\PP^\pm_j})=\rTr_{\PP^\pm_j}(\Id_{\PP^\pm_j}),\ \text{for $1\leq j\leq \rt$}.$$
\subsubsection*{Endomorphisms $x^\pm_j$}
For $x\in U$, let us denote by $l^V_x$ the operator of the left multiplication by $x$ on $V$. Sometimes, we will omit $V$ for simplicity. 

To compute $\rTr_{\PP_j^\pm}(x_j^\pm)$ we will use the action of the Casimir element $$C=FE+\frac{qK+q^{-1}K^{-1}}{(q-q^{-1})^2}.$$
Since $C$ is central, $l^\PP_C$ commutes with the left $U$-action, and hence defines an
  endomorphism in $\Upmod$.
 
On simple modules, $C$ acts by a scalar, hence 
\begin{equation}\label{E:traceActionCPp}
\rTr_{\PP_\rt^\pm}(l_C^{\Proj_\rt^\pm}) 
=\frac{2 (\mp 1)^{\rt} }{(q-q^{-1})^2}.
\end{equation} 

For $j\in\{1,\dots,\rt-1\}$,  the dominant vector $\junbp_0$ (resp. $\junym_0$) of $\Proj_j^+$ (resp. $\Proj_j^-$) has weight $\pm q^{j-1}$.   The action of $C$ on this vector is 
$$C\junbp_0=\junap_0+\frac{q^j+q^{-j}}{(q-q^{-1})^2}\junbp_0$$
(resp. $C\junym_0 = \junxm_0 -\frac{q^j+q^{-j}}{(q-q^{-1})^2}\junym_0$). Thus, for   $j\in\{1,\dots,\rt-1\}$ we have
\begin{equation}\label{E:ActionC}
l_C^{\Proj_j^\pm}=x^\pm_j\pm \frac{q^j+q^{-j}}{(q-q^{-1})^2}\Id_{\PP^\pm_{j}}\ .
\end{equation}

To compute the action of $C$ on tensor products we need
the following formula.
\begin{multline*}
\Delta(C)=K^{-1}\otimes FE+K^{-1}E\otimes FK\\+F\otimes E+FE\otimes K
 +\frac{qK\otimes K+q^{-1}K^{-1}\otimes K^{-1}}{(q-q^{-1})^2}\ .$$
\end{multline*}
The second and third terms of $\Delta(C)$ have no
diagonal contibution, and hence
vanish when computing the partial trace $\tr_r^{\repX^+_2}\left(l_{\Delta(C)}^{\Proj_i^\pm\otimes \repX_2^+}\right)$.
In particular, for $v\in\Proj_i^\pm$ we have
\begin{multline*}
\left[\tr_r^{\repX^+_2}\left(l_{\Delta(C)}^{\Proj_i^\pm\otimes \repX_2^+}\right)\right](v)\\
=-q^{-1}K^{-1}v  -(q^2+q^{-2})FEv -\frac{(q^2+q^{-2})q}{(q-q^{-1})^2}Kv + \frac{-2q^{-1}}{(q-q^{-1})^2}K^{-1}v.
\end{multline*}
When $i\in \{1,...,\rt-1\}$ and $v$ is the generating vector of $\Proj_i^\pm$ this equality implies 
$$
\tr_r^{\repX^+_2}\left(l_{\Delta(C)}^{\Proj_i^\pm\otimes \repX_2^+}\right)=-(q^2+q^{-2})x_i^\pm
\mp \frac{(q^2+q^{-2})}{(q-q^{-1})^2}(q^i+q^{-i})\Id_{\Proj_i^\pm}.
$$
Similarly, if $v$ is the highest weight of $  \Proj_\rt^\pm$, we obtain
\begin{equation}\label{E:prTrCp2}
\tr_r^{\repX^+_2}\left(l_{\Delta(C)}^{\Proj_{\rt}^\pm\otimes \repX_2^+}\right)=
\pm2 \frac{(q^2+q^{-2})}{(q-q^{-1})^2}\Id_{\Proj_\rt^\pm}.
\end{equation}

Now we can compute $\rTr_{\PP^\pm_{\rt-1}}(x_{\rt-1}^\pm)$.  From the isomorphism in \eqref{E:IsoPPp}, we have
$$
\rTr_{\PP^\pm_{\rt}\otimes \repX_2^+} \left(l_{\Delta(C)}^{\Proj_{\rt}^\pm\otimes \repX_2^+}\right)=\rTr_{\PP^\pm_{\rt-1}}\left(l_C^{\Proj_{\rt-1}^\pm}\right).
$$
Using Equations \eqref{E:ActionC} and \eqref{E:prTrCp2} we can simplify the last equality as follows
$$
\pm 2 \frac{q^2+q^{-2}}{(q-q^{-1})^2}\rTr_{\PP^\pm_{\rt}}(\Id_{\PP^\pm_{\rt}}^\pm)=\rTr_{\PP^\pm_{\rt-1}}(x_{\rt-1}^\pm)\pm \frac{q^{\rt-1}+q^{-\rt+1}}{(q-q^{-1})^2}\rTr_{\PP^\pm_{\rt-1}}(\Id_{\PP^\pm_{\rt-1}}^\pm)\ ,$$
or
$$\rTr_{\PP^\pm_{\rt-1}}(x_{\rt-1}^\pm)=\pm(\mp)^{\rt-1}  \frac{(q^{\rt-1}-q^{-\rt+1})^2}{(q-q^{-1})^2}.
 $$
Using the isomorphism in Equation \eqref{E:IsoPPj}, for $j\in\{2,\dots,\rt-2\}$ we 
 obtain the following recursive relation
\begin{equation}
\rTr_{\PP^\pm_{j-1}}(x_{j-1}^\pm)+\rTr_{\PP^\pm_{j+1}}(x_{j+1}^\pm)+(q^2+q^{-2})\rTr_{\PP^\pm_{j}}(x_{j}^\pm)=
-2(\pm 1)^{\rt}(-1)^j\ .
\end{equation}

Using Equation \eqref{E:IsoPP1} we can show that this formula also holds for $j=\rt-1$ by setting $x_{\rt}^\pm=0$.
We deduce the general formula for $j\in\{1,\dots,\rt-1\}$
 $$\rTr_{\PP^\pm_{j}}(x_{j}^\pm)=(\pm 1)^\rt (-1)^{j}[j]^2\ ,
$$ 
which is compatible with the computation at $j=\rt-2$ or $j=\rt-1$  and satisfies the recursive relation for $j\in\{2,\dots,\rt-2\}$.

With similar computation we get the same value for $\lTr_{\PP^\pm_{j}}(x_{j}^\pm) $.
 Thus, we have proved that the left and right modified traces are equal on $\Upmod$.
\hfill\mbox

Let us summarize our computations of the modified trace ($0<j<\rt$):
$$\begin{array}{|c|c|c|c|c|c|}
\hline
\Id_{\PP_\rt^-}& \Id_{\PP_\rt^+ }&\Id_{\PP_j^-}& \Id_{\PP_j^+}&x_j^-&x_j^+\\
\hline
1&(-1)^{\rt-1}&
(-1)^{\rt+j-1} (q^j+q^{-j})&(-1)^{j} (q^j+q^{-j})  &(-1)^{\rt+j} [j]^2&(-1)^j [j]^2\\ \hline
\end{array}$$.

\subsection*{Proof of Theorem \ref{main}}
Let us apply the modified trace construction to the regular representation $U\in \Upmod$.
Using the isomorphism of algebras
$$r:\Urest^{op}\cong\End_\Urest(\Urest), \;\; x\mapsto r_x$$
 where $r_x(y)=yx$ is the right multiplication,  
we
 define 
 $$\mTr:\Urest\to \C\quad\text{ by}\quad \mTr(x)=\t_\Urest(r_x).$$
By Theorem \ref{mTrace}, the linear map $\mTr$ is a character and satisfies the partial trace property.

The fact that the pairing between 
$$Z(\U) \times \HH(U) \to \C \quad\text{given by}\quad
(z,x)\mapsto \mTr(zx)$$ is non-degenerate
can now be shown by a direct computation.
Using Proposition \ref{P:Prop444} and Theorem \ref{mTrace}
we can explicitly compute
this pairing in the base of the center and the trace.
For example
\begin{equation*}
  \mTr\left(\idemw_{j}^+ \tclass^+_j\right)
 = 
  \rTr_{\PP^+_{j}}\left(  x_{j}^+\right)
   =(-1)^{j}[j]^2
\end{equation*}
or
 $$
   \mTr(\ideme_0\tclass^-_\rt)=\rTr_{\PP_\rt^-}( \Id_{\PP_\rt^-})
   =1.
$$
Completing the computation, we obtain
the pairing  shown in Table \ref{Table:Pairing}. 
\begin{table}[htp]
\caption{Values of the pairing on a basis, here $1\leq j<\rt$.}
\begin{center}
$\begin{array}{|c||c|c|c|c|c|}
\hline
&\tclass_\rt^+&\tclass_\rt^-&\tclass_s&\tclass_{s}^+&\tclass_{\rt-s}^-\\
\hline\hline
\ideme_\rt&(-1)^{\rt-1}&0&0&0&0
\\ \hline
\ideme_0&0&1&0&0&0
\\ \hline
\ideme_j&0&0&(-1)^j[j]^2&(-1)^j(q^j+q^{-j})&(-1)^j(q^j+q^{-j}) \\ \hline
\idemw_j^+&0&0&0&(-1)^j[j]^2&0 \\ \hline
\idemw_j^-&0&0&0&0&(-1)^j[j]^2\\ \hline
\end{array}$
\end{center}
\label{Table:Pairing}
\end{table}
From the table it is easy to see it is  non-degenerate.

\section{Proof of Theorem \ref{E:HenningsLog}}
In this section we will define our logarithmic 3-manifold
invariant $\Hlog(M,L)$,  prove Theorem \ref{E:HenningsLog} and 
compare $\Hlog(M,L)$ with the invariant defined by Jun Murakami
in \cite{jM2013}.
\subsection*{Logarithmic invariant}
Assume we are given a link $(L^+,L^-)$ with
$(m_+,m_-)$ components inside a $3$-manifold $M=S^3(L^0)$, where $L^0\subset S^3$ is 
a  surgery link for $M$ with $m_0$ components.
We suppose that $(L^+,L^-)$ is in $S^3\setminus L^0$, and choose a string link
$T=(T^+,T^0,T^-)$ whose closure is $(L^+,L^0,L^-)$.
By Lemma \ref{E:RTleftAction},
 the universal invariant $$J_T\in \left(\Urest^{\otimes (m_++m_0+m_-)}\right)^\Urest .$$
Let us color the components of $L^+$ and $L^-$  by central  elements $z_j\in Z(U)$, $1\leq j\leq m_+$ and by trace classes $h_k
\in \HH(U)$, $1\leq k\leq m_-$, respectively, and write
 $z^+=\otimes_j z_j$, $h^-=\otimes_k h_k$.
 Let us denote by $L$ the resulting colored link
 $((L^+,z^+), (L^-,h^-))$.
We obtain
\begin{equation*}
\Hlog(M, L):=\delta^s
\langle\left(   z^+\mu^{\otimes m_+}\otimes \mu^{\otimes m_0}\otimes \Id\right)\left(J_T\right),\,h^-\,
\rangle
\end{equation*}
by evaluating components of $J_T$ corresponding to
$(L^+,L^0)$ with
the right integral $\mu$ twisted with central elements and 
by applying to the result the modified trace pairing. 
Here $s$ is the signature of the linking matrix for $L^0$.
Theorem \ref{E:HenningsLog} claims that 
$\Hlog(M, L)$ is a topological invariant of the pair
$(M, L)$.

\begin{proof}[Proof of Theorem \ref{E:HenningsLog}]
We first show  that
$\Hlog(M, L)$ is an invariant of the colored link $L$, i.e.
it does not depend on the choice of $T$.
 By applying Lemma \ref{XXX} 
we see that $\Hlog(M, L)$ is invariant of the tangle $T_1$
obtained by closing the first $m_+ +m_0$ components of $T$.

Using the partial trace property  of the modified trace,
recursively, we
can safely close further $m_--1$ components of $T_1$
with quantum characters.
 Under this operation,
the colors $\tclass^\pm_j$ and $\tclass_j$  correspond to the quantum character
$\tr^{\PP^{\pm}_j}_l$  and this character pre-composed with $x^+_j$, respectively, for any $1\leq j< \rt$. The main point is that it does not matter which $m_- -1$ components we choose!

Hence
$\Hlog(M, L)$ is invariant of the $(1,1)$-tangle $K$ obtained from $T_1$ by closing all but one trace class colored components. 

But the resulting central element 
$J_K$ does not depend on the point where we cut $L$ into  $K$. The proof mimics the argument showing that long knots
and knots in $S^3$ are equivalent.

Indeed, let us think
about our $(1,1)$-tangle $K$ as being a long knot.
We need to show that our central element does not change if we move an arc through "infinity" (or the cutting disc). This move
can be alternatively realized by moving the arc in the opposite way through
the long knot, which is just  a sequence of Reidemeister moves, under which we know $J_K$ to be stable.


%
It remains to show  invariance under Kirby moves:
sliding along a component of $L^0$ and stabilisation with $\pm 1$ framed unknot. 
The defining property of the right integral ensure the sliding invariance (see e.g. \cite{Kerler_genealogy}).
Note that if we change the orientation on one of the
$T^0$ components, this will change $J_T$ by applying $S$ 
at the corresponding position, but now $\mu \circ S$ is a left integral, hence the sliding property holds after rearranging the components.

Adding to $L^0$ a $\pm 1$ framed unknot multiplies $\Hlog(M,L)$
 by $\mu(v^\mp)=\delta^{\mp}$,
and  changes the signature $s$ by $\pm 1$, so $\Hlog(M, L)$
remains the same.
\end{proof}


\subsection*{Relation with other invariants}
 
Here we show that  the logarithmic invariant
of Jun Murakami   \cite{jM2013}
is a special case of $\Hlog(M,L)$ where
$L^-$ has precisely one component.

Murakami's  invariant  is defined for a knot and a colored link in a $3$-manifold.
 We will adapt our notation  to his setting. 
Let us consider a link $(L^+,L^-)$ in a $3$-manifold $M=S^3(L_0)$
 as before, but with $L_-=K$  a knot
 and $m_-=1$. The link $L^+$ is colored by $z^+$ as before.
 The logarithmic knot invariant defined by Murakami is then
$$\Jun(M, (L^+,z^+), K)=\delta^s
(  z^+\mu^{\otimes m_+}\otimes \mu^{\otimes m_0}\otimes \Id)(J_T)
\in Z(U)\ .$$
 Due to different conventions in the definition of  $J_T$,  Murakami's original invariant rather corresponds to the opposite link
 in our notation.
 
Murakami further expands his invariant in the basis 
of the center as follows
$$\Jun(M,(L^+,z^+), K)= \sum^p_{j=0} a_j \ideme_j +
\sum^{p-1}_{j=1} b^+_j\idemw^+_j + \sum^{p-1}_{j=1} b^-_j\idemw^-_j.$$
The coefficients are clearly topological invariants of the 
triple $(M,(L^+,z^+),K)$.

\begin{proposition}
With the above notation, we have ($1\leq j<\rt$)
\begin{align*}
a_0&=\Hlog(M,(L^+,z^+),(L^-,\tclass_p^-))\\
 a_p&=(-1)^{\rt-1}\Hlog(M,(L^+,z^+),(L^-,\tclass_p^+))\\
 a_j&=\frac{(-1)^{j}}{[j]^2}\Hlog(M,(L^+,z^+),(L^-,\tclass_j))\\
 b_j^+&=\frac{(-1)^{j}}{[j]^4}\Hlog(M,(L^+,z^+),
(L^-,[j]^2\tclass_j^+-(q^j+q^{-j})\tclass_j))\\ 
b_j^-&=\frac{(-1)^{j}}{[j]^4}\Hlog(M,(L^+,z^+),
(L^-,[j]^2\tclass_{p-j}^--(q^j+q^{-j})\tclass_j))
\end{align*}
\end{proposition}
\begin{proof}
For any trace class $h\in \HH(\Urest)$ we have 
$$\Hlog((M,(L^+,z^+),(L^-,h))=\langle \Jun(M,(L^+,z^+),L^-),h\rangle \ .$$
From Table \ref{Table:Pairing} we get
\begin{eqnarray*}
\Hlog(M,(L^+,z^+),(L^-,\tclass_p^-)&=&a_0\\
\Hlog(M,(L^+,z^+),(L^-,\tclass_p^+)&=& (-1)^{\rt-1} a_p\\
\Hlog(M,(L^+,z^+),(L^-,\tclass_j)&=& {(-1)^{j}}{[j]^2}a_j\\
\Hlog(M,(L^+,z^+),
(L^-,\tclass_j^+)&=&  {(-1)^{j}}(q^j+q^{-j})a_j
+ {(-1)^{j}}{[j]^2}b^+_j\\
\Hlog(M,(L^+,z^+),
(L^-,\tclass_{p-j}^-)&=&  {(-1)^{j}}(q^j+q^{-j})a_j
+ {(-1)^{j}}{[j]^2}b^-_j\ .
\end{eqnarray*}
The claim follows.
\end{proof}

\end{document}